\documentclass[a4paper, 11pt]{amsart}
\usepackage[UKenglish]{babel}
\usepackage[utf8]{inputenc}
\usepackage{amsthm}
\usepackage{amsfonts}
\usepackage{mathtools}
\usepackage{enumitem} \setlist[enumerate]{label={\upshape(\arabic*)}}
\usepackage{amssymb}
\usepackage[numbers]{natbib}
\usepackage{url}
\usepackage{mathrsfs}
\usepackage{geometry}
\usepackage{bbm}
\usepackage{tikz-cd}
\usepackage{color}
\usepackage{framed}
\usepackage{hyperref}
\hypersetup{
    colorlinks=true, 
    linkcolor=blue, 
    urlcolor=red, 
    citecolor=[rgb]{0,0.7,0},
    linktoc=all 
}
\usepackage{titlesec}

\titleformat{\section}
  {\normalfont\scshape}{\thesection}{1em}{}

\titleformat{\subsection}[runin]
  {\normalfont\bfseries}{\thesubsection}{1em}{}

\calclayout
  
\theoremstyle{definition}
\newtheorem{defn}{Definition}[section]
\newtheorem{prop}[defn]{Proposition}

\newtheorem{thm}[defn]{Theorem}

\newtheorem{lem}[defn]{Lemma}
\newtheorem{cor}[defn]{Corollary}
\newtheorem{rmk}[defn]{Remark}

\newtheorem{conj}[defn]{Conjecture}

\numberwithin{equation}{section}

\title{THE CHOW RING OF HYPERKÄHLER VARIETIES OF $K3^{[2]}$-TYPE VIA LEFSCHETZ ACTIONS}
\author{Andreas Kretschmer}
\date{}
\address{Otto-von-Guericke University Magdeburg, Institute for Algebra and Geometry, Magdeburg, Germany}
\email{andreas.kretschmer@ovgu.de}

\begin{document}
\nocite{*}
\setlength{\parindent}{0em}

\begin{abstract}
We propose an explicit conjectural lift of the Neron--Severi Lie algebra of a hyperkähler variety $X$ of $K3^{[2]}$-type to the Chow ring of correspondences ${\rm CH}^\ast(X \times X)$ in terms of a canonical lift of the Beauville--Bogomolov class obtained by Markman. We give evidence for this conjecture in the case of the Hilbert scheme of two points of a $K3$~surface and in the case of the Fano variety of lines of a very general cubic fourfold. Moreover, we show that the Fourier decomposition of the Chow ring of $X$ from \cite[Theorem~2]{SV} agrees with the eigenspace decomposition of a canonical lift of the grading operator.
\end{abstract}

\maketitle

\section{Introduction and results}\label{Intro}

The main purpose of this note is to hint at a connection between two lines of reasoning which are both relevant for the study of the Chow ring of hyperkähler varieties of $K3^{[2]}$-type. One focuses on Hilbert schemes of points of $K3$~surfaces \cite{O, NOY} using Nakajima operators, and the other investigates an analog of the Fourier transform and injectivity results for the cycle class map on subrings of universal classes \cite{SV, FLSV}. The present text suggests a connection in two directions. The first is provided by Conjecture~\ref{MainConj}, asking for an extension of the main theorem of~\cite{O} to the $K3^{[2]}$-case. The second comes in the form of Theorem~\ref{agreeFourier}, stating that the Fourier decomposition of the Chow ring from \cite[Theorem~2]{SV} agrees with the eigenspace decomposition of a canonical lift of the cohomological grading operator $h$.

\subsection{Conventions.}
We denote the Chow ring of a smooth projective variety $X$ over $\mathbb{C}$ by ${\rm CH}^\ast(X)$, and it will always be with coefficients in $\mathbb{Q}$. Similarly, we often abbreviate the cohomology ring $H^\ast(X, \mathbb{Q})$ by $H^\ast(X)$. For a cycle class $Z \in {\rm CH}^\ast(X)$ or a cohomology class $\beta \in H^\ast(X)$ we denote the pullbacks to $X \times X$ via the two projections by $Z_1$, $Z_2$ and $\beta_1$, $\beta_2$, and similar conventions will be followed throughout.

\subsection{Lefschetz actions on the Chow ring.}
Let $X$ be a hyperkähler variety of \mbox{$K3^{[2]}$-type} and $(-,-)$ the Beauville--Bogomolov bilinear form on $H^2(X)$. Let $L \in {\rm CH}^2(X \times X)$ be Markman's canonical lift of the associated cohomology class and $l \in {\rm CH}^2(X)$ as in \eqref{eqldef}, see Section~\ref{pre} for details. We denote the second Betti number of $X$ by $r = 23$.

\begin{prop}
For every divisor class $a \in {\rm CH}^1(X)$ with $(a,a) \neq 0$ the cycle classes
\begin{gather*}
F_a := \frac{4}{(r+2)(a,a)} (l_1 a_1 + l_2 a_2) + \frac{2}{(a,a)}L(a_1 + a_2) \in {\rm CH}^3(X \times X), \\
H := \frac{4}{r (r+2)} (l_2^2 - l_1^2) + \frac{2}{r+2}L(l_2 - l_1) \in {\rm CH}^4(X \times X)
\end{gather*}
are lifts of the Lefschetz dual operator $f_a$ and the grading operator $h$, respectively.
\end{prop}

\begin{conj}\label{MainConj}
Let $\mathfrak{g}_{\rm NS}(X)$ be the Neron--Severi Lie algebra of $X$ and ${\rm cl}$ the cycle class map. The linear map $\varphi$ in the commutative diagram
\begin{center}
\begin{tikzcd}
\mathfrak{g}_{\text{NS}}(X) \ar[r, "\varphi"] \ar[dr, hook]
	& {\rm CH}^\ast(X \times X) \ar[d, "{\rm cl}"] \\
	& {\rm End}_{\mathbb{Q}}(H^\ast(X, \mathbb{Q})),
\end{tikzcd}
\end{center}
given by $\varphi(e_a) = \Delta_\ast(a)$, $\varphi(f_a) = F_a$ and $\varphi(h) = H$ is a well-defined Lie algebra homomorphism.
\end{conj}

In the case $X = S^{[2]}$ for a projective $K3$~surface~$S$ we use the explicit description from Theorem~\ref{LdefHilb} of Markman's lift $L$ in order to prove in Section~\ref{Hilb2} that our formulas for the lifts $F_a$ and $H$ agree with the canonical lifts provided in~\cite{O} in terms of Nakajima operators. This is the content of Proposition~\ref{T_a=f_a}. Hence, the main theorem of~\mbox{\emph{loc. cit.}} shows that Conjecture~\ref{MainConj} is true if $X = S^{[2]}$. In the case of the Fano variety of lines we use a result of Fu, Laterveer, Vial and Shen \cite{FLSV} and the explicit description of Markman's lift from Theorem~\ref{Ldef} to obtain the following partial confirmation of Conjecture~\ref{MainConj}.

\begin{thm}\label{LieChow}
Let $F = F(Y)$ be the Fano variety of lines of a smooth cubic fourfold~$Y$ and $g \in {\rm CH}^1(F)$ the Plücker polarization class. Let $L \in {\rm CH}^2(F \times F)$ be Markman's lift. Then there is a Lie algebra homomorphism
\begin{align*}
\mathfrak{sl}_2(\mathbb{Q}&) \rightarrow {\rm CH}^\ast(F \times F), \\
&e \mapsto \Delta_\ast(g), \\
&f \mapsto F_g, \\
&h \mapsto H,
\end{align*}
lifting the $\mathfrak{sl}_2(\mathbb{Q})$-action on cohomology given by the Lefschetz triple $(e_g, f_g, h)$.
\end{thm}

Whenever $F$ has Picard rank $1$, i.e., ${\rm CH}^1(F) = \langle g \rangle$, which is true for very general cubic fourfolds~$Y$, then Conjecture~\ref{MainConj} reduces to Theorem~\ref{LieChow}. We also propose two new relations in Conjecture~\ref{conjRel} which would yield a generalization of Theorem~\ref{LieChow} to the case of divisor classes different from $g$, see Proposition~\ref{LieChowGeneralization}. More evidence for Conjecture~\ref{MainConj} is provided by Remark~\ref{[Ta,Tb]=0}, where we note that $[F_a, F_b] = 0$ in full generality.\\
The obstacle to proving the conjecture in the case of Picard rank $>1$ is that we neither have at our disposal an analog of the machinery of Nakajima operators nor sufficiently strong injectivity results for the cycle class map yet, involving non-tautological divisor classes and hence extending results such as Theorem~\ref{inj}, which is the main geometrical input for Theorem~\ref{LieChow}.

\subsection{Eigenspace decomposition of $H$.}
If the lift $H$ of $h$ is diagonalizable, it could be expected to be multiplicative with respect to the intersection product. This is because the eigenspace decomposition of $H$ can be viewed as an analog of the Beauville decomposition in the abelian variety case, as discussed in the introduction of~\cite{NOY}.

\begin{thm}\label{eigspace}
Let $X$ be a hyperkähler variety of $K3^{[2]}$-type endowed with a lift $L \in {\rm CH}^2(X \times X)$ of $\mathfrak{B}$ satisfying the relations \eqref{eqquad}-\eqref{relation25}. Let $\Lambda_\lambda^i \subseteq {\rm CH}^i(X)$ be the eigenspace for the eigenvalue $\lambda$ of $H_\ast$ and denote ${\rm CH}^i(X)_s := \Lambda_{2i-4-s}^i$. The operator $H_\ast \in {\rm End}_\mathbb{Q}({\rm CH}^\ast(X))$ is diagonalizable with eigenspace decomposition
\begin{align}\label{eigDecomp}
\begin{split}
{\rm CH}^0(X) &= {\rm CH}^0(X)_0, \\
{\rm CH}^1(X) &= {\rm CH}^1(X)_0, \\
{\rm CH}^2(X) &= {\rm CH}^2(X)_0 \oplus {\rm CH}^2(X)_2, \\
{\rm CH}^3(X) &= {\rm CH}^3(X)_0 \oplus {\rm CH}^3(X)_2, \\
{\rm CH}^4(X) &= {\rm CH}^4(X)_0 \oplus {\rm CH}^4(X)_2 \oplus {\rm CH}^4(X)_4.
\end{split}
\end{align}
All direct summands outside the leftmost column belong to the homologically trivial cycle classes ${\rm CH}^\ast(X)_{\rm hom}$. We have
\begin{equation*}
{\rm CH}^3(X)_2 = \Lambda^3_0 = {\rm CH}^3(X)_{\rm hom}, \ {\rm CH}^4(X)_0 = \Lambda^4_4 = \langle l^2 \rangle, \ {\rm CH}^4(X)_2 = \Lambda^4_2 = l \cdot L_\ast({\rm CH}^4(X)),
\end{equation*}
so that the cycle class map is injective on ${\rm CH}^i(X)_0$ except maybe for $i=2$. Moreover, all elements of ${\rm CH}^4(X)_0 \oplus {\rm CH}^4(X)_2$ are multiples of $l$, and multiplication by $l$ gives an injective map ${\rm CH}^1(X) \rightarrow {\rm CH}^3(X)_0$. Furthermore, $L_\ast({\rm CH}^4(X)_4) = 0$.
\end{thm}

The number $s$ should be seen as a sort of defect, e.g., the terms with $s=0$ give the expected eigenvalue $2i-4$ of $H_\ast$. The eigenspace decomposition only shows terms with $s \geq 0$. In fact, there is a conjecture by Beauville in the abelian variety case which predicts this behavior, see \cite{BeauSplit}. It is also expected that the cycle class map is injective on ${\rm CH}^\ast(X)_0$, and Theorem~\ref{eigspace} confirms this in all codimensions except~$2$.

\begin{conj}\label{multiplicativityConj}
Let $X$ be as in Theorem~\ref{eigspace} such that $L$ additionally satisfies \eqref{relationL^2}. For all occurring $s, t \in \mathbb{Z}$ we conjecture that the intersection product gives a well-defined map
\begin{equation*}
{\rm CH}^i(X)_s \times {\rm CH}^j(X)_t \overset{\cdot}{\longrightarrow} {\rm CH}^{i+j}(X)_{s+t}.
\end{equation*}
\end{conj}

This can be rewritten in several ways. For this and a generalization in the case of Hilbert schemes of points of $K3$~surfaces see~\cite[eq.~(6)~and~(44)]{NOY}. Evidence for the conjecture is provided by Corollary~\ref{Abstauber} below.

\subsection{The Fourier decomposition.}\label{FourierSection}
We now compare the eigenspace decomposition of $H_\ast$ in Theorem~\ref{eigspace} to the Fourier decomposition of \cite[Theorem~2]{SV} which needs the additional relation \eqref{relationL^2}. The Fourier transform is given by the correspondence
\begin{equation*}
e^L = [X \times X] + L + \frac{1}{2} L^2 + \frac{1}{6} L^3 + \frac{1}{24} L^4 \in {\rm CH}^\ast(X \times X).
\end{equation*}
We define the Fourier decomposition groups by
\begin{equation*}
{^{e^L}{\rm CH}^i}(X)_s := \{Z \in {\rm CH}^i(X): (e^L)_\ast(Z) \in {\rm CH}^{4-i+s}(X) \}.
\end{equation*}
They do not depend on the precise values of the coefficients of the powers $L^i$ as long as they are non-zero.

\begin{thm}[{\cite[Theorem~2]{SV}}]\label{FourierDecomp}
Let $X$ be a hyperkähler variety of $K3^{[2]}$-type endowed with a lift $L \in {\rm CH}^2(X \times X)$ of $\mathfrak{B}$ satisfying all the relations \eqref{eqquad}-\eqref{relationL^2}. Then the decomposition \eqref{eigDecomp} equally holds with ${\rm CH}^i(X)_s$ replaced by ${{^{e^L}{\rm CH}^i}(X)_s}$.
\end{thm}

\begin{thm}\label{agreeFourier}
Let $X$ be a hyperkähler variety of $K3^{[2]}$-type endowed with a lift $L \in {\rm CH}^2(X \times X)$ of $\mathfrak{B}$ satisfying all the relations \eqref{eqquad}-\eqref{relationL^2}. Then the eigenspace decomposition \eqref{eigDecomp} of $H_\ast$ agrees with the Fourier decomposition of Theorem~\ref{FourierDecomp}.
\end{thm}

The only place where the relation \eqref{relationL^2} is needed is the existence of the Fourier decomposition of ${\rm CH}^2(X)$, see the proof of \cite[Theorem~2.4]{SV}.

\begin{cor}\label{Abstauber}
Assumptions as in Theorem~\ref{agreeFourier}. In the eigenspace decomposition of Theorem~\ref{eigspace} all occuring direct summands are non-trivial. Moreover, Conjecture~\ref{multiplicativityConj} is true if $X = S^{[2]}$ is Hilbert scheme of two points of a projective $K3$~surface~$S$ and if $X = F(Y)$ is the Fano variety of lines of a \emph{very general} cubic fourfold~$Y$ with $L$ being Markman's lift in both cases. 
\end{cor}

\begin{proof}
For the non-vanishing of every occurring direct summand, see \cite[p.~7]{SV}. For the multiplicativity in the $S^{[2]}$ case we even have two independent possibilities. One is \cite[Theorem~3]{SV} and the other is \cite[Theorem~1.4]{NOY} together with Proposition~\ref{T_a=f_a} showing that the two canonical lifts of the cohomological grading operator $h$ agree. The case of a Fano variety of lines $F(Y)$ of a very general cubic foufold~$Y$ is again \cite[Theorem~3]{SV}.
\end{proof}

\begin{rmk}\label{rmkAbstauber}
It becomes clear from the proof of \cite[Theorem~3]{SV} in the case of a Fano variety $F = F(Y)$ that the assumption on $Y$ to be very general is only needed for the inclusions
\begin{align*}
{\rm CH}^1(F)_0 \cdot {\rm CH}^2(F)_0 \subseteq {\rm CH}^3(F)_0,\\
{\rm CH}^2(F)_0 \cdot {\rm CH}^2(F)_0 \subseteq {\rm CH}^4(F)_0,
\end{align*}
and the first one was dealt with for $Y$ not necessarily very general in \cite[Proposition~A.7]{FLSV}. The second inclusion, however, still remains open for arbitrary smooth cubic fourfolds~$Y$. In the very general case these inclusions are even equalities.
\end{rmk}

\subsection*{Acknowledgments.}
I want to thank my advisor Georg Oberdieck for proposing the topic of my master's thesis -- from which the present work originated -- and for his constant encouragement and supply of ideas, as well as Daniel Huybrechts for a very helpful suggestion during my talk in the master's thesis seminar. Moreover, I had the great opportunity to attend a talk given by Mingmin Shen in December 2019, and I am very thankful for his kind and patient answers to my questions.

\section{Preliminaries}\label{pre}

\subsection{The Neron--Severi Lie algebra.}\label{NeronSeveri}
Let $X$ be a smooth projective variety over $\mathbb{C}$ of complex dimension $n$. An element $a \in H^2(X, \mathbb{Q})$ is called \emph{Lefschetz} if the conclusion of the hard Lefschetz theorem holds, i.e., for the operator $e_a$ of cup product by $a$ the $k$-fold iteration $e_a^k: H^{n-k}(X) \rightarrow H^{n+k}(X)$ is an isomorphism for all $1 \leq k \leq n$. The operator $e_a$ is called the \emph{Lefschetz operator}. If $a$ is Lefschetz, there exists a unique operator $f_a$, the \emph{Lefschetz dual}, such that the commutator $[e_a, f_a] = h$ is the \emph{grading operator}, given by multiplication with the integer $k-n$ in degree $k$. The triple $(e_a, f_a, h)$ satisfies the $\mathfrak{sl}_2$-commutation relations and is called a \emph{Lefschetz triple}. Looijenga and Lunts~\cite{LL} and Verbitsky~\cite{Ver} introduced the \emph{total Lie algebra} $\mathfrak{g}(X) \subseteq {\rm End}_{\mathbb{Q}}(H^\ast(X))$ of $X$ which is generated by all Lefschetz triples $(e_a, f_a, h)$. The \emph{Neron--Severi Lie algebra} of $X$ is the Lie subalgebra $\mathfrak{g}_{\text{NS}}(X) \subseteq \mathfrak{g}(X)$ generated by only those Lefschetz triples where $a$ is algebraic, i.e., $a \in H^{1,1}(X, \mathbb{Q})$. In view of the Grothendieck standard conjecture of Lefschetz type, a natural question to ask, then, is whether the Neron--Severi Lie algebra action on the cohomology ring can be lifted to an action on the Chow ring or, slightly stronger, whether there is a Lie algebra homomorphism $\mathfrak{g}_{\text{NS}}(X) \rightarrow {\rm CH}^\ast(X \times X)$ to the ring of correspondences ${\rm CH}^\ast(X \times X)$ such that
\begin{center}
\begin{tikzcd}
\mathfrak{g}_{\text{NS}}(X) \ar[r] \ar[dr, hook]
	& {\rm CH}^\ast(X \times X) \ar[d, "{\rm cl}"] \\
	& {\rm End}_{\mathbb{Q}}(H^\ast(X, \mathbb{Q}))
\end{tikzcd}
\end{center}
commutes, where ${\rm cl}$ denotes the cycle class map. By the main theorem of~\cite{O}, this is the case if $X$ is the Hilbert scheme of points of a projective $K3$~surface. We investigate this question for the Fano variety of lines of a smooth cubic fourfold, obtaining Theorem~\ref{LieChow} as a partial analog. We view the Fano variety of lines as a test case for Conjecture~\ref{MainConj}.

\subsection{Markman's lift of the Beauville--Bogomolov class.}\label{MarkmanLift}
Let $X$ be a hyperkähler variety and $(-,-): H^2(X, \mathbb{Q}) \times H^2(X, \mathbb{Q}) \rightarrow \mathbb{Q}$ the non-degenerate symmetric bilinear form associated to the Beauville--Bogomolov quadratic form on $X$ \cite{HuyInv}. Via the Künneth isomorphism we obtain from this the Beauville--Bogomolov cohomology class $\mathfrak{B} \in H^4(X \times X, \mathbb{Q})$. Extending the coefficients to $\mathbb{C}$ and choosing an orthonormal basis $(e_i)$ of $H^2(X, \mathbb{C})$ with respect to $(-,-)$, we can write $\mathfrak{B} = \sum_{i=1}^r e_i \otimes e_i$, where $r = b_2(X)$ is the second Betti number. We denote by $\mathfrak{b} := \Delta^\ast(\mathfrak{B}) \in H^4(X, \mathbb{Q})$ the pullback along the diagonal embedding $\Delta: X \hookrightarrow X \times X$. Over $\mathbb{C}$ we can write $\mathfrak{b} = \sum_{i=1}^r e_i^2$.

\begin{thm}[{\cite{Markman}}] \label{existenceL}
Let $X$ be a hyperkähler variety deformation equivalent to the Hilbert scheme of $n \geq 2$ points of a $K3$ surface. Then there exists a lift $L \in {\rm CH}^2(X \times X)$ of the Beauville--Bogomolov class $\mathfrak{B}$.
\end{thm}

A summary of Markman's construction of $L$ in \cite{Markman} is given in \cite[Section 1.9]{SV}, particularly Theorem~9.15 where Shen and Vial use the $\kappa$-class $\kappa_2(M) \in {\rm CH}^2(X \times X)$ of \emph{Markman's sheaf} $M$, a twisted sheaf in the sense of \cite[Definition~2.1]{Markman}, in order to define $L$. The twisted sheaf $M$ is constructed as in the theorem below, following closely \cite[Theorem~9.12]{SV}.

\begin{thm}[\cite{Markman}]\label{MarkmanSheaf}
Let $X$ be a hyperkähler manifold of $K3^{[n]}$-type. There exists a $K3$ surface $S$ and a suitable Mukai vector $v$ with a $v$-generic ample divisor $H$ such that there is a proper flat family $\pi: \mathcal{X} \rightarrow C$ having the following properties:
\begin{enumerate}
\item The curve $C$ is connected but possibly reducible of arithmetic genus $0$.
\item There exist $t_1, t_2 \in C$ such that $\mathcal{X}_{t_1} = X$ and $\mathcal{X}_{t_2}= M_H(v)$, the latter denoting the moduli space of stable sheaves on $S$ with Mukai vector $v$.
\item There is a twisted universal sheaf $\mathcal{E}$ on $M_H(v) \times S$.
\item There is a torsion-free reflexive coherent twisted sheaf $\mathcal{G}$ on $\mathcal{X} \times_C \mathcal{X}$, flat over $C$, satisfying $\mathcal{G}_{t_2} \cong \mathcal{E}xt^1_{p_{13}}(\mathcal{E}_{12}, \mathcal{E}_{32})$ on $M_H(v) \times M_H(v)$.
\end{enumerate}
\end{thm}

Then set \emph{Markman's sheaf} to be the twisted sheaf $M := \mathcal{G}_{t_1}$ on $X \times X$. \emph{A priori}, $M$ might depend on the chosen deformation, but as it turns out, $\kappa_2(M)$ does not. As in \cite{Markman2} we denote by $\mathfrak{M}_\Lambda$ the moduli space of isomorphism classes of marked hyperkähler manifolds.

\begin{prop}[\cite{Markman2}]\label{propMarkmanSheaf}
The $\kappa$-class $\kappa_2(M)$ is independent of the chosen deformation and starting point $M_H(v) \in \mathfrak{M}_\Lambda$.
\end{prop}

In the paragraph preceding Definition~6.16 of \cite{Markman} the notion of a \emph{parametrized twistor path} $\gamma: C \rightarrow \mathfrak{M}_\Lambda$ is defined. We stress that Markman requires $\gamma$ to map the irreducible components of the possibly reducible curve $C$ isomorphically onto twistor lines in the moduli space $\mathfrak{M}_\Lambda$. Now, by \cite[Theorem~1.11]{Markman2}, the Azumaya algebra $\mathcal{E}nd(M)$ (see \cite[Definition~1.1]{Markman2}) associated with $M$ is independent of the twistor path and also of the starting point $M_H(v) \in \mathfrak{M}_{\Lambda}$ up to possibly dualizing $M$. Hence, $\kappa_2(M)$ is independent of both as it is a rational multiple of the second Chern class of $\mathcal{E}nd(M)$, see \cite[Lemma~2.4]{Markman}. Shen and Vial now define first
\begin{equation}\label{eqldef}
l := \frac{1}{{\rm deg}(c_{2n}(X))} (p_2)_\ast \Big( \kappa_2(M) \cdot (p_1)^\ast(c_{2n}(X)) \Big) + c_2(X) \in {\rm CH}^2(X),
\end{equation}
where $p_i: X \times X \rightarrow X$ are the projections and $c_i(X)$ are the Chern classes of the tangent bundle of $X$. This is indeed a lift of $\mathfrak{b}$ by \cite[Lemma~9.14]{SV}. They then set
\begin{equation}\label{eqLdef}
L := - \kappa_2(M) + \frac{1}{2} (p_1)^\ast(l - c_2(X)) + \frac{1}{2} (p_2)^\ast(l - c_2(X)) \in {\rm CH}^2(X \times X),
\end{equation}
which is indeed a lift of $\mathfrak{B}$ by the same lemma. By Proposition~\ref{propMarkmanSheaf}, $L$ and $l$ are, in fact, canonical lifts of their cohomology classes, and we are justified in calling $L$ \emph{Markman's lift} of $\mathfrak{B}$ without ambiguity. The equation $\Delta^\ast(L) = l$ holds if and only if $l$ is a rational multiple of $c_2(X)$ in ${\rm CH}^2(X)$.

\subsection{Markman's lift in terms of tautological classes for the Hilbert scheme.}\label{ExplicitLiftHilb}
In the case of the Hilbert scheme $S^{[2]}$ of two points of a projective $K3$~surface~$S$ we have an explicit description of Markman's lift of $\mathfrak{B}$, see \cite{SV} for details.\\
Denote $F = S^{[2]}$ and let $\mathcal{Z} \subseteq F \times S$ be the universal family. Its set of closed points consists of the pairs $(\eta, x)$ where $x \in {\rm supp}(\eta)$. This is a codimension $2$ closed subscheme of the product, and we denote by
\begin{center}
\begin{tikzcd}
\mathcal{Z} \ar[drr, bend left, "q"] \ar[ddr, bend right, "p"'] \ar[dr, hook] \\

	& F \times S \ar[r, "\rho"] \ar[d, "\pi"']
		& S \\

	& F
\end{tikzcd}
\end{center}
the projections. Let $c \in {\rm CH}_0(S)$ be the canonical $0$-cycle, represented by any point on a rational curve in $S$ \cite[Theorem~1]{BV}. We let
\begin{equation}
S_c := p_\ast q^\ast(c) \in {\rm CH}^2(F).
\end{equation}
Moreover, denote by $\Delta_{\rm Hilb} \in {\rm CH}^1(F)$ the divisor class on $F$ parametrizing the non-reduced length $2$ subschemes of $S$ and set
\begin{equation}
\delta := \frac{1}{2} \Delta_{\rm Hilb} \in {\rm CH}^1(F).
\end{equation}
This agrees with the convention of \cite{SV} which differs in the sign from \cite{O}. Finally, let $I \subseteq F \times F$ be the subset of pairs of length $2$ subschemes which share a common support point. This is closed and irreducible (see the proof of \cite[Lemma~11.2]{SV}), and endowed with the reduced induced subscheme structure it gives a closed subvariety of codimension~$2$, called the \emph{incidence subscheme}. Its cycle class in ${\rm CH}^2(F \times F)$ is also denoted $I$ and called the \emph{incidence correspondence}. By \mbox{\emph{loc. cit.}} we have $I = {^t \mathcal{Z}} \circ \mathcal{Z}$.

\begin{thm}[{\cite[p.~67]{SV}}]\label{LdefHilb}
Let $F = S^{[2]}$ for a projective $K3$ surface $S$. An explicit lift $L$ of the Beauville--Bogomolov class $\mathfrak{B}$ is given by
\begin{equation}\label{LHilb}
L = I - 2 (S_c)_1 - 2 (S_c)_2 - \frac{1}{2} \delta_1 \delta_2.
\end{equation}
Moreover, by \cite[Proposition~16.1]{SV}, $L$ agrees with Markman's lift. For its pullback along the diagonal embedding $i_{\Delta} : F \hookrightarrow F \times F$ we have
\begin{equation}\label{lHilb}
l := i_{\Delta}^\ast(L) = 20 S_c - \frac{5}{2} \delta^2.
\end{equation}
In fact, $l = \frac{5}{6} c_2(T_F)$ where $c_2(T_F)$ is the second Chern class of the tangent bundle.
\end{thm}

\subsection{Markman's lift in terms of tautological classes for the Fano variety.}\label{ExplicitLiftFano}

For the Fano variety of lines $F := F(Y) \subseteq {\rm Gr}(2,6)$ of a smooth cubic fourfold $Y$ we define several tautological cycle classes as follows.

\begin{defn}
	Let $\mathcal{E}$ be the tautological bundle on ${\rm Gr}(2,6)$. Let $g \in {\rm CH}^1(F(Y))$ be the first Chern class $c_1(\mathcal{E}|_{F(Y)}^\vee)$. Then $g$ is called the \emph{Plücker polarization class}. Moreover, we denote by $c \in {\rm CH}^2(F(Y))$ the second Chern class $c_2(\mathcal{E}|_{F(Y)}^\vee)$.
\end{defn}

\begin{defn}
	The \emph{incidence subscheme} is the closed subset $I \subseteq F \times F$ with the reduced subscheme structure, given by the set of pairs of intersecting lines inside $Y$. Its cycle class in ${\rm CH}^2(F \times F)$, also denoted $I$, is called the \emph{incidence correspondence}. The \emph{tautological subring} $R^\ast(F \times F) \subseteq {\rm CH}^\ast(F \times F)$ is the $\mathbb{Q}$-subalgebra generated by $I, \Delta, c_1, c_2, g_1, g_2$ where $\Delta \subseteq F \times F$ denotes the diagonal and $g_i$, $c_i$ for $i = 1, 2$ are the pullbacks via the two projections $F \times F \rightarrow F$.
\end{defn}

\begin{thm}[{\cite[p. 81]{SV}}]\label{Ldef}
	Let $Y$ be a smooth cubic fourfold and $F = F(Y)$ its Fano variety of lines. An explicit lift $L$ of the Beauville--Bogomolov class $\mathfrak{B}$ is given by
	\begin{equation}\label{eqLdefFano}
	L = \frac{1}{3}(g_1^2 + \frac{3}{2}g_1g_2 + g_2^2 - c_1 - c_2) - I.
	\end{equation}
	Moreover, we have $l := \Delta^\ast(L) = \frac{5}{6} c_2(T_F)$, where $c_2(T_F)$ is the second Chern class of the tangent bundle, and $c_2(T_F) = 5g^2 - 8c$. Hence, the tautological subring $R^\ast(F \times F)$ contains $L, l_1, l_2$. By the quadratic equation \eqref{eqquad} below for $L$, $R^\ast(F \times F)$ can in fact be generated by $L, l_1, l_2, g_1, g_2$.
\end{thm}

\begin{rmk}
	Let $M$ be Markman's sheaf on $F(Y) \times F(Y)$ as in Theorem~\ref{MarkmanSheaf}. By \cite[Corollary~1.6]{Markman3} and \cite{Markman2}, the Azumaya algebra $\mathcal{E}nd(M)$ is universally defined over the moduli space of Fano varieties. Hence, by \cite[Theorem~1.10]{FLSV}, the explicit lift of $\mathfrak{B}$ from \eqref{eqLdefFano}, which clearly is universally defined as well, agrees with Markman's lift \eqref{eqLdef}.
\end{rmk}

The following result from \cite{FLSV} is crucial for the proof of Theorem~\ref{LieChow}.

\begin{thm}[{\cite[Proposition~6.4]{FLSV}}] \label{inj}
Let $F = F(Y)$ be the Fano variety of lines of a smooth cubic fourfold $Y$. Then the restriction of the cycle class map to
\begin{equation*}
{\rm cl}: R^\bullet(F \times F) \rightarrow H^{2 \bullet}(F \times F)
\end{equation*}
is injective.
\end{thm}

In fact, Shen and Vial in \cite{SV} prove the following relations for the explicit lifts $L$ of Theorems~\ref{LdefHilb} and~\ref{Ldef}. The first is referred to as the \emph{quadratic equation for~$L$}:
\begin{equation}\label{eqquad}
L^2 = 2 \Delta - \frac{2}{r+2}L(l_1 + l_2) - \frac{1}{r(r+2)}(2l_1^2 - rl_1 l_2 + 2l_2^2).
\end{equation}
Additionally, we consider the following three relations for all $\sigma \in {\rm CH}^4(X)$ and all $\tau \in {\rm CH}^2(X)$:
\begin{align}
L_\ast(l^2) &= 0, \label{relationL(l^2)=0} \\
L_\ast(l \cdot L_\ast(\sigma)) &= (r+2) L_\ast(\sigma), \label{relation25} \\
(L^2)_\ast(l \cdot (L^2)_\ast(\tau)) &= 0. \label{relationL^2}
\end{align}

These are the relations which Shen and Vial in~\cite[Theorem~2]{SV} established to be the core relations necessary to obtain a Fourier decomposition. The lifts from Theorems~\ref{LdefHilb} and~\ref{Ldef} indeed satisfy all of these, see Theorem~1 and the paragraphs after Theorem~2 of \mbox{\emph{loc. cit.}} In the case of the Fano variety of lines they all follow more directly from cohomology, Theorem~\ref{inj} and the fact that the tautological subring $R^\ast(F \times F)$ is closed under the composition of correspondences. The latter is a consequence of \cite[Proposition~6.3]{FLSV}.

\section{Proofs of the main results} \label{Fano}

\subsection{Hyperkähler fourfolds.}
Let $X$ be a hyperkähler variety of complex dimension~$4$ without odd cohomology over~$\mathbb{Q}$ or, equivalently, $H^3(X) = 0$. Then the cup product is commutative and we always write $\alpha \beta$ instead of $\alpha \cup \beta$. By $(e_i)$ we denote an arbitrary orthonormal basis of $H^2(X, \mathbb{C})$ with respect to the Beauville--Bogomolov form and we denote the second Betti number by $r := b_2(X)$. The (modified) Fujiki constant $c_X$ equals $1$ for hyperkähler varieties of $K3^{[n]}$-type. Finally, we write $\mathbf{1} \in H^8(X)$ for the generator of $H^8(X)$ with integral $1$.

\begin{prop}\label{f_a}
Let $a \in H^2(X, \mathbb{Q})$ and define $\widetilde{f}_a: H^\bullet(X) \rightarrow H^{\bullet -2}(X)$ by
\begin{equation*}
\widetilde{f}_a(\beta) := \begin{cases}  0 & \beta \in H^0(X), \\ 4(a, \beta) [X] & \beta \in H^2(X), \\ \frac{2}{c_X} \mathfrak{B}_\ast(a \beta) & \beta \in H^4(X), \\ \frac{2}{c_X} \mathfrak{B}_\ast(\beta) & \beta \in H^6(X), \\ \frac{4}{(r+2)c_X} \left( \int_X \beta \right) \mathfrak{b} a & \beta \in H^8(X). \end{cases}
\end{equation*}
Then $\widetilde{f}_a$ satisfies $[e_a, \widetilde{f}_a] = (a,a)h$ and $[h, \widetilde{f}_a] = -2 \widetilde{f}_a$. Moreover, $[h, e_a] = 2e_a$.
\end{prop}

If $(a,a) \neq 0$, we set $f_a := \frac{1}{(a,a)} \widetilde{f}_a$ to obtain the usual $\mathfrak{sl}_2(\mathbb{Q})$-commutation relations. By Lemma~\ref{e_a^2} below, $(a,a) \neq 0$ is equivalent to $a$ being Lefschetz in the sense of Subsection~\ref{NeronSeveri}.

\begin{rmk}
If $(a,a) \neq 0$, then $\widetilde{f}_a$ is uniquely determined by the commutation relations while this fails if $(a,a) = 0$ in which case the zero map also satisfies them. It should be emphasized also that $\widetilde{f}_a$ is linear in $a \in H^2(X)$ because for one, $f_a$ is not, and neither is it clear from the abstract description of $f_a$ that it would suffice to multiply it by some quadratic form $(a,a)$ in order to make it linear in $a$.
\end{rmk}

Before proving Proposition~\ref{f_a}, we need some lemmas. They are simple computations in cohomology using the definition of the Beauville--Bogomolov form, so we omit the proofs.

\begin{lem}\label{bb}
For arbitrary $\gamma, \gamma' \in H^2(X)$ we have $\int_X \mathfrak{b} \gamma \gamma' = c_X (r+2) (\gamma, \gamma')$.
\end{lem}

In the same vein and using that $(-,-)$ is non-degenerate, one proves:

\begin{lem}\label{BB}
The linear map $H^2(X) \rightarrow H^6(X)$ given by the cup product with $\frac{1}{(r+2)c_X} \mathfrak{b}$ is an isomorphism with inverse $\mathfrak{B}_\ast$.
\end{lem}

\begin{lem}\label{e_a^2}
Let $a \in H^2(X, \mathbb{Q})$. Then the following are equivalent:
\begin{enumerate}
\item $(a,a) \neq 0$,
\item $e_a^2: H^2(X) \rightarrow H^6(X)$ is an isomorphism,
\item $e_a^4: H^0(X) \rightarrow H^8(X)$ is an isomorphism.
\end{enumerate}
In particular, $a$ is Lefschetz if and only if $(a,a) \neq 0$.
\end{lem}

\begin{proof}
The equation $\int_X a^4 = 3 c_X (a,a)^2$ proves that $(a,a) \neq 0$ if and only if $e_a^4$ is an isomorphism. Next, for $(a,a) \neq 0$ let $\beta \in H^2(X)$ be in the kernel of $e_a^2$, i.e., $a^2 \beta = 0$. It suffices to show $\beta = 0$ by Poincaré duality. Now, if $a^2 \beta = 0$, then $a^2 \beta \gamma = 0$ for every $\gamma \in H^2(X)$. Setting $\gamma = a$ gives $0 = \int_X a^3 \beta = 3 c_X (a,a)(a, \beta)$, hence $(a,\beta) = 0$. Then for arbitrary $\gamma$ we get $0 = \int_X a^2 \beta \gamma = c_X (a,a) (\beta, \gamma) + 2 c_X (a, \beta) (a, \gamma) = c_X (a,a) (\beta, \gamma)$, so $\beta = 0$ as $(-,-)$ is non-degenerate. Conversely, let $e_a^2$ be an isomorphism. Then certainly $a \neq 0$, and so $0 \neq e_a^2(a) = a^3$. By Poincaré duality, there is some $\beta \in H^2(X)$ with $a^3 \beta \neq 0$, so $0 \neq \int_X a^3 \beta = 3 c_X (a,a)(a,\beta)$, in particular $(a,a) \neq 0$.
\end{proof}

\begin{proof}[Proof of Proposition~\ref{f_a}]
We show $[e_a, \widetilde{f}_a] = (a,a)h$ case by case. For $\beta \in H^0(X)$ assume by linearity $\beta = [X]$, in which case indeed
\begin{equation*}
[e_a, \widetilde{f}_a]([X]) = 0 - \widetilde{f}_a(a) = -4(a,a)[X] = (a,a)h([X]).
\end{equation*}
For $\beta \in H^2(X)$, we first compute
\begin{align*}
\mathfrak{B}_\ast(a^2 \beta) &= (p_2)_\ast \left( \sum_{i=1}^{r} (a^2 \beta e_i) \otimes e_i \right) \\ &= \sum_{i=1}^{r} \left(\int_X a^2 \beta e_i \right) e_i \\ &= c_X \sum_{i=1}^{r} \big((a,a)(\beta,e_i) + 2(a, \beta)(a,e_i)\big) e_i \\ &= c_X (a,a)\beta + 2 c_X (a,\beta)a.
\end{align*}
Hence,
\begin{equation*}
[e_a, \widetilde{f}_a](\beta) = a \widetilde{f}_a(\beta) - \widetilde{f}_a(a \beta) = 4(a,\beta)a - \frac{2}{c_X} \mathfrak{B}_\ast(a^2 \beta) = -2(a,a) \beta = (a,a)h(\beta),
\end{equation*}
as desired. For $\beta \in H^4(X)$, we have $[e_a, \widetilde{f}_a](\beta) = \frac{2}{c_X} a \mathfrak{B}_\ast(a \beta) - \frac{2}{c_X} a \mathfrak{B}_\ast(a \beta) = 0$. In the case $\beta \in H^6(X)$, by Lemma~\ref{BB}, we can write $\beta = \frac{1}{r+2} \mathfrak{b} \widehat{\beta}$ for a unique $\widehat{\beta} \in H^2(X)$. As $h(\beta) = 2 \beta$ we need to show $[e_a, \widetilde{f}_a](\beta) -2(a,a)\beta = 0$, which, by Poincaré duality, is equivalent to $\int_X \left( [e_a, \widetilde{f}_a](\beta) -2(a,a)\beta \right)\gamma = 0$ for all $\gamma \in H^2(X)$. Indeed,
\begin{align*}
\int_X \left( [e_a, \widetilde{f}_a](\beta) -2(a,a)\beta \right)\gamma &= \int_X \left( 2a^2 \widehat{\beta} \gamma \right) \\ &- \int_X \left(\frac{4}{(r+2)^2} \left( \int_X \mathfrak{b} \widehat{\beta} a \right) \mathfrak{b}a \gamma \right) - \int_X \left( \frac{2(a,a)}{r+2} \mathfrak{b} \widehat{\beta} \gamma \right) \\ &= c_X \left(2(a,a)(\widehat{\beta},\gamma) + 4(a, \widehat{\beta})(a,\gamma)\right) \\ &- 4 c_X (a,\widehat{\beta})(a,\gamma) - 2 c_X (a,a)(\widehat{\beta},\gamma) \\ &= 0.
\end{align*}
Finally, let $\beta \in H^8(X)$. Here,
\begin{equation*}
[e_a, \widetilde{f}_a](\beta) = a \widetilde{f}_a(\beta) - 0 = \frac{4}{(r+2)c_X} \left( \int_X \beta \right) \mathfrak{b} a^2 = 4(a,a) \beta = (a,a) h(\beta),
\end{equation*}
using Lemma~\ref{bb} and the fact that $H^8(X)$ is of dimension $1$. The relations $[h, \widetilde{f}_a] = -2 \widetilde{f}_a$ and $[h, e_a] = 2e_a$ follow directly from the fact that $\widetilde{f}_a$ decreases and $e_a$ increases the degree by $2$.
\end{proof}

\begin{cor}\label{T_a}
Let $X$ be a hyperkähler variety of complex dimension~$4$ with $H^3(X, \mathbb{Q}) = 0$. Let $\mathfrak{B}$ admit a lift $L \in {\rm CH}^2(X \times X)$ and $\mathfrak{b}$ a lift $l \in {\rm CH}^2(X)$. Then the cycle class
\begin{equation*}
\widetilde{F}_a := \frac{4}{(r+2)c_X}(l_1 a_1 + l_2 a_2) + \frac{2}{c_X} L (a_1 + a_2) \in {\rm CH}^3(X \times X)
\end{equation*}
is a lift of $\widetilde{f}_a$. If $(a,a) \neq 0$, we again set $F_a := \frac{1}{(a,a)} \widetilde{F}_a$.
\end{cor}

\begin{defn}
Under the assumptions of Corollary~\ref{T_a}, and for a fixed lift $L$ of $\mathfrak{B}$, we denote the commutator by
\begin{equation*}
\widetilde{H}_a := [\Delta_\ast(a), \widetilde{F}_a] = (a_2 - a_1) \widetilde{F}_a
\end{equation*}
and $H_a := \frac{1}{(a,a)} \widetilde{H}_a$ for $(a,a) \neq 0$. Automatically, $\widetilde{H}_a$ is a lift of $(a,a)h$ where $h$ is the cohomological grading operator.
\end{defn}

\begin{prop}\label{relCohom}
Let $a \in H^2(X, \mathbb{Q})$. Then, in $H^8(X \times X, \mathbb{Q})$, the following relation holds:
\begin{equation*}
(a,a) \mathfrak{B} \mathfrak{b}_1 = (r+2) \mathfrak{B} a_1^2 - 2 \mathfrak{b}_1 a_1 a_2.
\end{equation*}
\end{prop}

\begin{proof}
First observe that it suffices to show
\begin{equation*}
\left( (r+2) \mathfrak{B} a_1^2 - 2 \mathfrak{b}_1 a_1 a_2 - (a,a) \mathfrak{B} \mathfrak{b}_1 \right) \gamma_1 = 0
\end{equation*}
for all $\gamma \in H^2(X)$ by Poincaré duality. A direct computation then shows indeed
\begin{align*}
{\rm LHS} = \ &(r+2) \mathfrak{B} a_1^2 \gamma_1 - 2 \mathfrak{b}_1 a_1 a_2 \gamma_1 - (a,a) \mathfrak{B} \mathfrak{b}_1 \gamma_1 \\ = \ &(r+2) \sum_{i=1}^r (e_i a^2 \gamma) \otimes e_i - 2 \sum_{i=1}^r (e_i^2 a \gamma) \otimes a - (a,a) \sum_{i,j=1}^r (e_i e_j^2 \gamma) \otimes e_i \\ = \ &(r+2)c_X \cdot \mathbf{1} \otimes \left( \sum_{i=1}^r e_i \big( 2 (a,\gamma)(a,e_i) + (a,a)(\gamma, e_i) \big) \right) \\ - \ &2 c_X \cdot \mathbf{1} \otimes a \cdot \sum_{i=1}^r \big( (a,\gamma) + 2(a, e_i)(\gamma, e_i) \big) \\ - \ &c_X (a,a) \cdot \mathbf{1} \otimes \left( \sum_{i,j=1}^r e_i \big( (\gamma, e_i) + 2\delta_{ij} (\gamma, e_j) \big) \right) \\ = \ &2(r+2)c_X (a, \gamma) \cdot \mathbf{1} \otimes a + (r+2) c_X (a,a) \cdot \mathbf{1} \otimes \gamma \\ - \ &2(r+2) c_X (a, \gamma) \cdot \mathbf{1} \otimes a \\ - \ &(r+2) c_X (a,a) \cdot \mathbf{1} \otimes \gamma \\ = \ &0.
\end{align*}
\end{proof}

An easy consequence of Proposition~\ref{relCohom} is the next relation.

\begin{prop}\label{relCohom2}
For all $a \in H^2(X, \mathbb{Q})$ we have $r \mathfrak{B} \mathfrak{b}_1 a_1 = \mathfrak{b}_1^2 a_2$ in $H^{10}(X \times X, \mathbb{Q})$.
\end{prop}

In the same vein, one shows:

\begin{prop}\label{relCohom3}
For all $a \in H^2(X, \mathbb{Q})$ the following relation holds in $H^{10}(X \times X, \mathbb{Q})$:
\begin{equation*}
(r+2) \mathfrak{B}^2 a_1 = 2\mathfrak{B} \mathfrak{b}_1 a_2 + \mathfrak{b}_1 \mathfrak{b}_2 a_1.
\end{equation*}
\end{prop}

\begin{prop}\label{H}
Let $X$ be a hyperkähler variety of complex dimension~$4$ with vanishing $H^3(X)$. Whenever $L \in {\rm CH}^2(X \times X)$ is a lift of $\mathfrak{B}$ and $l \in {\rm CH}^2(X)$ a lift of $\mathfrak{b}$ then
\begin{equation*}
H := \frac{4}{r(r+2)}(l_2^2 - l_1^2) + \frac{2}{r+2}(l_2 - l_1)L \in {\rm CH}^4(X \times X)
\end{equation*}
is a lift of the grading operator $h$.
\end{prop}

\begin{proof}
Transposing the relation from Proposition~\ref{relCohom} and subtracting the two equations yields
\begin{equation*}
(a,a)\mathfrak{B}(\mathfrak{b}_2 - \mathfrak{b}_1) = (r+2) \mathfrak{B}(a_2^2 - a_1^2) -2a_1a_2(\mathfrak{b}_2 - \mathfrak{b}_1).
\end{equation*}
Now,
\begin{align*}
\frac{r+2}{2}[\widetilde{H}_a] &= \frac{r+2}{2} (a_2-a_1) [\widetilde{F}_a] \\ &= (r+2) \mathfrak{B} (a_2^2 - a_1^2) - 2a_1 a_2 (\mathfrak{b}_2 - \mathfrak{b}_1) + 2(\mathfrak{b}_2 a_2^2 - \mathfrak{b}_1 a_1^2) \\ &= (a,a) \mathfrak{B} (\mathfrak{b}_2 - \mathfrak{b}_1) + 2(\mathfrak{b}_2 a_2^2 - \mathfrak{b}_1 a_1^2) \\ &= (a,a) \mathfrak{B} (\mathfrak{b}_2 - \mathfrak{b}_1) + \frac{2}{r} (a,a) (\mathfrak{b}_2^2 - \mathfrak{b}_1^2) \\ &= \frac{r+2}{2}(a,a) [H].
\end{align*}
We used the above equation in line three and Lemma~\ref{bb} in line four. Hence, $\widetilde{H}_a$ and $(a,a)H$ agree in cohomology, and therefore $H$ is a lift of $h$ as long as there exists some $a$ with $(a,a) \neq 0$ which is always the case for formal reasons.
\end{proof}

The following proof is inspired by the proof of \cite[Theorem~2.2]{SV}.

\begin{proof}[Proof of Theorem~\ref{eigspace}]
In codimensions $0$ and $1$ the cycle class map injects into cohomology. Also, $L_\ast({\rm CH}^i(X)) \subseteq {\rm CH}^{i-2}(X)$ and $(l_1^2)_\ast$ acts only on ${\rm CH}^0(X)$ while $(l_2^2)_\ast$ acts only on ${\rm CH}^4(X)$. Now, for $Z \in {\rm CH}^2(X)$ we have $L_\ast(Z) = 0$ by cohomology. Hence, $H_\ast(Z) = \frac{-2}{r+2} L_\ast(lZ)$. By the hypotheses, $L_\ast(l \cdot L_\ast(\sigma)) = (r+2) L_\ast(\sigma)$ for all $\sigma \in {\rm CH}^4(X)$, so
\begin{align*}
((H_\ast + 2 {\rm id}) \circ H_\ast) (Z) &= (H_\ast + 2{\rm id})\left(\frac{-2}{r+2} L_\ast(lZ) \right) \\ &= \frac{4}{(r+2)^2} L_\ast(l \cdot L_\ast(lZ)) - \frac{4}{r+2} L_\ast(lZ) = 0,
\end{align*}
giving the decomposition of ${\rm CH}^2(X)$. Next, let $Z \in {\rm CH}^3(X)$. Then $H_\ast(Z) = \frac{2}{r+2} l L_\ast(Z)$, and $L_\ast(Z)$ is a divisor class. Therefore and by Lemma~\ref{BB}, $L_\ast(Z) = 0$ if and only if $Z \in {\rm CH}^3(X)_{\rm hom}$ is homologically trivial. Again by Lemma~\ref{BB}, in cohomology the cup product by $\frac{1}{r+2} [l]$ is the inverse isomorphism of $[L]_\ast: H^6(X) \rightarrow H^2(X)$. Hence, if $H_\ast(Z) = 0$, then in cohomology
\begin{equation*}
0 = [H_\ast(Z)] = \frac{2}{r+2}[l] \cup [L_\ast(Z)] = 2[Z],
\end{equation*}
hence $Z \in {\rm CH}^3(X)_{\rm hom}$. We have shown $\Lambda^3_0 = {\rm CH}^3(X)_{\rm hom}$. Next, by the quadratic equation for $L$, we have $((L^2)_\ast - 2 {\rm id})(Z) = \frac{-2}{r+2} l L_\ast(Z) = -H_\ast(Z)$, or equivalently,
\begin{equation*}
(H_\ast - 2 {\rm id}) (Z) = - (L^2)_\ast(Z).
\end{equation*}
But here, $(L^2)_\ast(Z) \in {\rm CH}^3(X)_{\rm hom} = \Lambda^3_0$ for degree reasons. Therefore, $H_\ast \circ (H_\ast - 2 {\rm id}) = 0$ on ${\rm CH}^3(X)$, giving the desired decomposition. In order to see that $l \cdot D$ is an element of $\Lambda^3_2$ for every divisor class $D$, just note $H_\ast(l \cdot D) = \frac{2}{r+2} l \cdot L_\ast(l D)$, and $L_\ast(l D)$ is a divisor which in cohomology agrees with $(r+2)D$. At last, let $Z \in {\rm CH}^4(X)$. Then
\begin{equation*}
H_\ast(Z) = \frac{4}{r(r+2)} \left( \int_X [Z] \right) l^2 + \frac{2}{r+2} l L_\ast(Z).
\end{equation*}
Now, $\int_X [H_\ast(Z)] = 4 \int_X [Z]$, implying
\begin{equation*}
H_\ast(H_\ast(Z)) = \frac{16}{r(r+2)} \left( \int_X [Z] \right) l^2 + \frac{4}{r+2} l L_\ast(Z),
\end{equation*}
in particular $(H_\ast^2 - 4 H_\ast)(Z) = \frac{-4}{r+2} l L_\ast(Z)$. Applying $H_\ast$ once again yields
\begin{equation*}
(H_\ast \circ (H_\ast^2 - 4 H_\ast))(Z) = \frac{-8}{r+2} l L_\ast(Z) = 2 (H_\ast^2 - 4 H_\ast)(Z),
\end{equation*}
using that $\int_X [l L_\ast(Z)] = 0$. This is because $L_\ast(Z) \in {\rm CH}^2(X)_{\rm hom}$ for degree reasons. We have seen, then, that $H_\ast \circ (H_\ast^2 - 4 H_\ast) = 2 (H_\ast^2 - 4 H_\ast)$, or equivalently,
\begin{equation*}
H_\ast \circ (H_\ast - 2 {\rm id}) \circ (H_\ast - 4 {\rm id}) = 0,
\end{equation*}
giving the desired decomposition for ${\rm CH}^4(X)$. Finally, let $Z \in \Lambda^4_4$. We want to show that $Z$ is a multiple of $l^2$. Indeed,
\begin{equation*}
4Z = H_\ast(Z) = l^2 \cdot \frac{4}{r(r+2)} \int_X [Z] + \frac{2}{r+2} l \cdot L_\ast(Z).
\end{equation*}
Applying $L_\ast$ to the equation and using again the hypothesis $L_\ast(l \cdot L_\ast(\sigma)) = (r+2) L_\ast(\sigma)$ for all $\sigma \in {\rm CH}^4(X)$, we get
\begin{equation*}
4 L_\ast(Z) = L_\ast(l^2) \cdot \frac{4}{r(r+2)} \int_X [Z] + \frac{2}{r+2} L_\ast( l \cdot L_\ast(Z)) = 2 L_\ast(Z),
\end{equation*}
hence $L_\ast(Z) = 0$. But then, the previous equation implies that $Z$ is a multiple of $l^2$. A similar argument shows $L_\ast(\Lambda^4_0) = 0$. The equation $\Lambda^4_2 = l \cdot L_\ast({\rm CH}^4(X))$ is immediate from the explicit formula for $H$ after observing $\Lambda^4_2 \subseteq {\rm CH}^4(X)_{\rm hom}$.
\end{proof}

\begin{proof}[Proof of Theorem~\ref{agreeFourier}]
In codimensions $0$ and $1$ there is nothing to show. For codimension~$2$ it suffices to show only the two inclusions
\begin{equation*}
{^{e^L}{\rm CH}^2}(X)_0 \subseteq {\rm CH}^2(X)_0 \ \text{ and } \ {^{e^L}{\rm CH}^2}(X)_2 \subseteq {\rm CH}^2(X)_2,
\end{equation*}
because in both cases their sum is ${\rm CH}^2(X)$, and the sums are direct. Let $Z \in {^{e^L}{\rm CH}^2}(X)_0$, i.e., $(e^L)_\ast(Z) \in {\rm CH}^2(X)$ which is easily seen to be equivalent to $(L^3)_\ast(Z) = 0$. Now, applying the quadratic equation for $L$ twice, it can be checked that
\begin{align*}
L^3 &= \frac{2r}{r+2} \Delta_\ast(l) + \frac{r+10}{(r+2)^2} L l_1 l_2 + {\rm cst}_1 \cdot L(l_1^2 + l_2^2) + {\rm cst}_2 \cdot (l_1^2 l_2 + l_1 l_2^2),
\end{align*}
where we used $L \cdot \Delta = L \cdot \Delta_\ast([X]) = \Delta_\ast(\Delta^\ast(L)) = \Delta_\ast(l)$. Applying this to $Z$ gives
\begin{align*}
0 = (L^3)_\ast(Z) &= \frac{2r}{r+2} l \cdot Z + \frac{r+10}{(r+2)^2} l \cdot L_\ast(l \cdot Z) + {\rm cst} \cdot l^2.
\end{align*}
Applying $L_\ast$ then yields $L_\ast(l \cdot Z) = 0$ by the relations $L_\ast(l^2) = 0$ and $L_\ast(l \cdot L_\ast(\sigma)) = (r+2) L_\ast(\sigma)$ for all $\sigma \in {\rm CH}^4(X)$. But now, the explicit formula for $H$ yields $H_\ast(Z) = \frac{2}{r+2} l \cdot L_\ast(Z) - \frac{2}{r+2} L_\ast(l \cdot Z) = 0$. For the second inclusion let $Z \in {^{e^L}{\rm CH}^2}(X)_2$ which is now equivalent to $(L^2)_\ast(Z) = 0$. Using the quadratic equation for $L$, we get
\begin{equation*}
0 = (L^2)_\ast(Z) = 2Z - \frac{2}{r+2} L_\ast(l \cdot Z) + \frac{1}{r+2} \left( \int_X [l \cdot Z] \right) l.
\end{equation*}
Now, by \cite[Proposition~4.1]{SV}, $l \cdot Z \in {^{e^L}{\rm CH}^4}(X)_2$, and the latter by \cite[Theorem~4]{SV} agrees with $l \cdot L_\ast({\rm CH}^4(X)) \subseteq {\rm CH}^4(X)_{\rm hom}$, hence $\int_X [l \cdot Z] = 0$. We thus obtain $L_\ast(l \cdot Z) = (r+2) Z$, and this is precisely equivalent to $H_\ast(Z) = 2Z$. In codimension $3$, let first $Z \in {^{e^L}{\rm CH}^3}(X)_2$. This is equivalent to $L_\ast(Z) = 0$. But now, $H_\ast(Z) = \frac{2}{r+2} l \cdot L_\ast(Z)$, so ${^{e^L}{\rm CH}^3}(X)_2 = {\rm CH}^3(X)_2$, as desired. Next, $Z \in {^{e^L}{\rm CH}^3}(X)_0$ is equivalent to $(L^2)_\ast(Z) = 0$, i.e., by the quadratic equation for~$L$,
\begin{equation*}
0 = (L^2)_\ast(Z) = 2Z - \frac{2}{r+2} l \cdot L_\ast(Z),
\end{equation*}
hence $H_\ast(Z) = \frac{2}{r+2} l \cdot L_\ast(Z) = 2Z$, concluding the codimension $3$ case. In codimension $4$, by \cite[Theorem~4]{SV}, we already know the equality of the direct summands
\begin{align*}
{^{e^L}{\rm CH}^4}(X)_0 &= \langle l^2 \rangle = \Lambda^4_4 = {\rm CH}^4(X)_0, \\
{^{e^L}{\rm CH}^4}(X)_2 &= l \cdot L_\ast({\rm CH}^4(X)) = \Lambda^4_2 = {\rm CH}^4(X)_2.
\end{align*}
Finally, let $Z \in {^{e^L}{\rm CH}^4}(X)_4$ which is equivalent to the vanishing of both $L_\ast(Z)$ and $\left( \int_X [Z] \right) [X]$. Thus,
\begin{equation*}
H_\ast(Z) = \frac{4}{r(r+2)} \left( \int_X [Z] \right) l^2 + \frac{2}{r+2} l \cdot L_\ast(Z) = 0,
\end{equation*}
i.e., $Z \in \Lambda^4_0 = {\rm CH}^4(X)_4$.
\end{proof}

\begin{rmk}\label{[Ta,Tb]=0}
Let $X$ be a hyperkähler variety of complex dimension $4$ with $H^3(X) = 0$ and $L \in {\rm CH}^2(X \times X)$ any lift of $\mathfrak{B}$ as well as $l \in {\rm CH}^2(X)$ any lift of $\mathfrak{b}$. Then for all divisor classes $a, b \in {\rm CH}^1(X)$ we have $[\widetilde{F}_a, \widetilde{F}_b] = 0.$ We omit the proof as it is a straightforward formal computation, although a little lengthy if spelled out in detail. Just use that all terms symmetric in $a$ and $b$ cancel out and that $(p_{13})_\ast \left( L_{12} \cdot L_{23} \cdot a_2 \right)$ is a divisor on $X \times X$ vanishing in cohomology.
\end{rmk}

\subsection{Fano variety of lines.}

\begin{proof}[Proof of Theorem~\ref{LieChow}]
By Corollary~\ref{T_a}, $F_g$ lifts $f_g$, and by Proposition~\ref{H}, $H$ lifts $h$. Moreover, $\Delta_\ast(g)$ lifts $e_g$. By the explicit formulas for $\Delta_\ast(g) = \Delta g_1$, $F_g$ and $H$ we know that all of them lie in the tautological subring $R^\ast(F \times F)$ of Theorem~\ref{inj}. As a consequence of \cite[Proposition~6.3]{FLSV}, so do their compositions. Hence all the commutators lie in $R^\ast(F \times F)$, and as the commutation relations are true in cohomology, they are true in the Chow ring as well.
\end{proof}

\begin{conj}\label{conjRel}
Let $F$ be the Fano variety of lines of a smooth cubic fourfold and let $L$ be Markman's lift. We conjecture the following relations in the Chow ring for all divisor classes $a \in {\rm CH}^1(F)$ with $(a,a) \neq 0$:
\begin{align}
(a,a)L l_1 &= (r+2)L a_1^2 - 2 l_1 a_1 a_2, \\
r L l_1 a_1 &= l_1^2 a_2, \\
(r+2) L^2 a_1 &= 2L l_1 a_2 + l_1 l_2 a_1.
\end{align}
\end{conj}

\begin{rmk}\label{conjRelRmk}
The second relation follows from the first one by \cite[Theorem~1.4]{Voisin}. Moreover, the first relation and its transpose would imply the important equality $H = H_a$ by the proof of Proposition~\ref{H} carried out in the Chow ring, replacing $\mathfrak{B}$ by $L$ and $\mathfrak{b}$ by $l$. This equality is necessary for establishing Conjecture~\ref{MainConj}. Note also that we indirectly show $H = H_a$ in the Hilbert scheme setting in Section~\ref{Hilb2}, see Remark~\ref{lastRmk}. In the Hilbert scheme setting, the above relations can be checked by computations similar to those in Section~\ref{comp}, and we expect them to hold true.
\end{rmk}

\begin{prop}\label{propRel}
Conjecture~\ref{conjRel} is true in the case of the Fano variety if $a$ is a multiple of the Plücker polarization class~$g$. In particular, $H = H_g$.
\end{prop}

\begin{proof}
By the aforegoing Propositions~\ref{relCohom}--\ref{relCohom3} the relations are true in cohomology for all $a$. If $a$ is a multiple of $g$, then by Theorem~\ref{inj} they hold in the Chow ring as well because all occurring terms then lie in the tautological subring $R^\ast(F \times F)$. Moreover, $(g,g) \neq 0$ by \cite[Section~2]{Ottem}.
\end{proof}

Under Conjecture~\ref{conjRel}, we have the following generalization of Theorem~\ref{LieChow}.

\begin{prop}\label{LieChowGeneralization}
Let $X$ be a hyperkähler variety of $K3^{[2]}$-type endowed with a symmetric lift $L \in {\rm CH}^2(X \times X)$ satisfying the relations \eqref{eqquad} and \eqref{relationL(l^2)=0}. Let $a \in {\rm CH}^1(X)$ be a divisor class satisfying the relations of Conjecture~\ref{conjRel} with respect to $L$. Then the analogous version of Theorem~\ref{LieChow} holds with respect to $L$ and for $g$ replaced by $a$.
\end{prop}

\begin{proof}
We need to show the three necessary commutation relations
\begin{equation*}
[H, \Delta_\ast(a)] = 2 \Delta_\ast(a), \ \ [H,F_a] = -2 F_a, \ \text{ and } \ [\Delta_\ast(a),F_a] = H.
\end{equation*}
For the last of these we already indicated this in Remark~\ref{conjRelRmk}. For the first one, observe that the left hand side equals
\begin{equation*}
(a_1 - a_2) H = \frac{4}{r(r+2)} (l_1^2 a_2 + l_2^2 a_1) + \frac{2}{r+2} L(l_1 a_2 + l_2 a_1) - \frac{2}{r+2}L(l_1 a_1 + l_2 a_2).
\end{equation*}
For the right hand side use the quadratic equation for $L$ and the second and third relation of the conjecture. Then
\begin{equation*}
2 \Delta_\ast(a) = \frac{2}{r+2} L(l_1 a_2 + l_2 a_1) + \frac{2}{r(r+2)} (l_1^2 a_1 + l_2^2 a_2).
\end{equation*}
Hence the difference equals $\frac{2}{r(r+2)}(l_1^2 a_2 + l_2^2 a_1) - \frac{2}{r+2} L (l_1 a_1 + l_2 a_2) = 0$ by another application of the second relation of the conjecture and its transpose. For the second commutation relation we show without loss of generality $[H, \widetilde{F}_a] = -2 \widetilde{F}_a$. Here, consider first the composition $H \circ \widetilde{F}_a$. The only summand of the latter which is not easily computed using the projection formula, $L_\ast(l^2) = 0$ and the second relation of Conjecture~\ref{conjRel} is
\begin{equation*}
-\frac{4}{r+2} (p_{13})_\ast \Big( (p_{13})^\ast(L l_2 a_1) \cdot (p_{23})^\ast(L) \Big),
\end{equation*}
where $p_{ij}: X \times X \times X \rightarrow X \times X$ denote the projections to the factors according to the indices. Here, we use the transpose of the third relation of the conjecture in order to write $L l_2 a_1 = \frac{r+2}{2} L^2 a_2 - \frac{1}{2} l_1 l_2 a_2$. From this we deduce that the above summand equals
\begin{equation*}
2 l_1 a_2 - 2 (p_{13})_\ast \Big( (p_{13})^\ast(L^2) \cdot (p_{23})^\ast(L a_1) \Big).
\end{equation*}
Next, using the quadratic equation for $L^2$ we obtain that the latter agrees with $-4 L a_1$. We obtain $H \circ \widetilde{F}_a = -\frac{16}{r+2} l_1 a_1 + \frac{8}{r+2} l_2 a_2 - 4 L a_1$. Now, $^tH = - H$ and $^t \widetilde{F}_a = \widetilde{F}_a$, so
\begin{equation*}
[H, \widetilde{F}_a] = -4 L(a_1 + a_2) - \frac{8}{r+2}(l_1 a_1 + l_2 a_2) = -2 \widetilde{F}_a,
\end{equation*}
as desired.
\end{proof}

\subsection{Hilbert scheme of two points of a $K3$ surface.} \label{Hilb2}

Let $S$ be a smooth projective surface. Recall the Nakajima operators on the Chow ring of Hilbert schemes of points of $S$,
\begin{equation*}
\mathfrak{q}_i: \bigoplus_{n \in \mathbb{Z}} {\rm CH}^\ast(S^{[n]}) \rightarrow \bigoplus_{n \in \mathbb{Z}} {\rm CH}^\ast(S^{[n+i]} \times S).
\end{equation*}
We refer to \cite{Nakajima, Groj} and to \cite[Section~2.3]{O} for a very brief overview. Recall that $\mathfrak{q}_i$ for $i \geq 1$ is given as a correspondence by the reduced closed subscheme
\begin{equation*}
Z_{n, n+i} := \{(\xi, x, \eta) \in S^{[n]} \times S \times S^{[n+i]}: {\rm supp}(\mathcal{I}_\xi/\mathcal{I}_\eta) = \{x\}\}
\end{equation*}
of the product $S^{[n]} \times S \times S^{[n+i]}$, where $\mathcal{I}_\xi$ is the ideal sheaf corresponding to the length $n$ closed subscheme $\xi$ of $S$. Recall also that $\mathfrak{q}_{-i}$ is just the transpose of $(-1)^i \mathfrak{q}_i$. Let $\mathcal{Z}_n \subseteq S^{[n]} \times S$ be the universal family. We write $\mathcal{Z} := \mathcal{Z}_2$.

\begin{lem}\label{projZ}
Let $n \geq 1$ and $p_{S \times S^{[n]}}: S^{[n-1]} \times S \times S^{[n]} \rightarrow S \times S^{[n]}$ be the projection. Then in ${\rm CH}^\ast(S \times S^{[n]})$ we have
\begin{equation*}
(p_{S \times S^{[n]}})_\ast \left( Z_{n-1,n} \right) = {^t \mathcal{Z}_n},
\end{equation*}
where ${^t \mathcal{Z}_n}$ is the transpose of $\mathcal{Z}_n$.
\end{lem}

\begin{proof}
First recall that $Z_{n-1,n}$ is a closed subvariety of the product. This can be seen using the projection map $Z_{n-1,n} \rightarrow S^{[n-1,n]}$ onto the nested Hilbert scheme \cite[Section~1.2]{Lehn2} and the fact that $S^{[n-1,n]}$ is irreducible by \cite[Theorem~1.9]{Lehn2}. Indeed, the projection is proper, hence closed, and it actually is a bijection on closed points. It follows that $Z_{n-1,n}$ is irreducible, hence a closed subvariety as it is endowed with the reduced subscheme structure.\\
Now, to prove the claim, note that the equation is clearly true set-theoretically. As $Z_{n-1,n}$ is a closed subvariety of the product, it suffices to show that the degree of $p_{S \times S^{[n]}}|_{Z_{n-1,n}}$ is~$1$. Considering a general fiber suffices, and if $(x, \eta) \in {^t \mathcal{Z}_n}$ with $\eta$ supported at $n$ distinct points $x, x_2, \ldots, x_n$ then the only preimage in $Z_{n-1,n}$ is $([x_2, \ldots, x_n], x, \eta)$ with multiplicity~$1$.
\end{proof}

\begin{cor}\label{INakajima}
As correspondences in ${\rm CH}^\ast(S^{[2]} \times S^{[2]})$ we have
\begin{equation*}
{^t \mathcal{Z}} \circ \mathcal{Z} = - \mathfrak{q}_1([S]) \circ \mathfrak{q}_{-1}([S])
\end{equation*}
in terms of Nakajima operators.
\end{cor}

\begin{proof}
This follows from the fact that the Hilbert--Chow morphism $S^{[1]} \rightarrow S$, sending a length~$1$ subscheme to its support point, is an isomorphism, and from the definition of $\mathfrak{q}_{\pm 1}([S])$ as correspondences in ${\rm CH}^\ast(S^{[2]} \times S^{[1]})$. Indeed,
\begin{align*}
\mathfrak{q}_1([S]) &= (p_{S^{[1]} \times S^{[2]}})_\ast \left( Z_{1,2} \right), \\
\mathfrak{q}_{-1}([S]) &= (-1) \cdot (p_{S^{[1]} \times S^{[2]}})_\ast \left( Z_{1,2} \right).
\end{align*}
The operator $\mathfrak{q}_{-1}([S])$ is then a correspondence from $S^{[2]}$ to $S^{[1]}$ while $\mathfrak{q}_1([S])$ is a correspondence in the opposite direction. Now, using $S^{[1]} = S$ by the Hilbert--Chow morphism, Lemma~\ref{projZ} for $n=2$ gives the desired result.
\end{proof}

\subsection{Computations.}\label{comp}
Let $S$ be a projective $K3$~surface and $\Delta_S \subseteq S \times S$ the diagonal, $\Delta_{123} \subseteq S \times S \times S$ the small diagonal. By a result of Beauville and Voisin \cite{BV} we have the following relations in the Chow ring of $S$, $S \times S$ and $S \times S \times S$ for every divisor class~$a$:
\begin{gather}\label{Identities}
c_2(T_S) = 24c, \\
\Delta_S c_1 = \Delta_S c_2 = c_1 c_2, \\
\Delta_S^2 = 24 c_1 c_2, \\
\Delta_S a_1 = a_1 c_2 + c_1 a_2, \\
\Delta_{123} = \Delta_{12} c_3 + \Delta_{13} c_2 + \Delta_{23} c_1 - c_1 c_2 - c_1 c_3 - c_2 c_3.
\end{gather}
The third equation is a consequence of the first one and the self-intersection formula. Now we show that the canonical lifts $F_a$ of the Lefschetz dual $f_a$ and $H$ of the grading operator $h$ agree with Oberdieck's canonical lifts in~\cite{O}. Recall the notation set up in Subsection~\ref{ExplicitLiftHilb}.

\begin{prop}\label{mult}
Let $S$ be a projective $K3$~surface and $F = S^{[2]}$. We have the following expressions for correspondences in terms of Nakajima operators:
\begin{align}
[F \times F] &= \frac{1}{4} \mathfrak{q}_1 \mathfrak{q}_1 \mathfrak{q}_{-1} \mathfrak{q}_{-1} ([S^4]), \\
{\rm mult}_{S_c} &= - \mathfrak{q}_1 \mathfrak{q}_{-1} (c_1 c_2) - \mathfrak{q}_2 \mathfrak{q}_{-2} (c_1 c_2), \\
{\rm mult}_{\delta^2} &= 12 \mathfrak{q}_2 \mathfrak{q}_{-2} (c_1 c_2) - \frac{1}{2} \mathfrak{q}_1 \mathfrak{q}_1 \mathfrak{q}_{-1} \mathfrak{q}_{-1} (\Delta_{1234}), \\
{\rm mult}_l &= -20 \mathfrak{q}_1 \mathfrak{q}_{-1} (c_1 c_2) - 50 \mathfrak{q}_2 \mathfrak{q}_{-2} (c_1 c_2) + \frac{5}{4} \mathfrak{q}_1 \mathfrak{q}_1 \mathfrak{q}_{-1} \mathfrak{q}_{-1}(\Delta_{1234}).
\end{align}
Of course, as correspondences we interpret ${\rm mult}_Z$ to be $(\Delta_F)_\ast(Z)$.
\end{prop}

\begin{proof}
The first of these equations follows immediately from Corollary~\ref{GeorgsTrick} and $[F] = \frac{1}{2} \mathfrak{q}_1([S]) \mathfrak{q}_1([S]) \cdot 1_{S^{[0]}}$. The last equation for ${\rm mult}_l$ of course follows from the two preceeding ones together with~\eqref{lHilb}. For the second equation we observe $S_c = p_\ast q^\ast(c) = \pi_\ast( {\rm ch}_2(\mathcal{O}_\mathcal{Z}) \rho^\ast(c))$, as $\mathcal{Z} = {\rm ch}_2(\mathcal{O}_\mathcal{Z})$ by \cite[eq.~(101)~on~p.~75]{SV}. Then the first formula of \cite[Theorem~1.6]{MN} shows the claim. Finally, for ${\rm mult}_{\delta^2}$ we use the operator $e_\delta$ of multiplication by $\delta$ and its expression in Nakajima operators from \cite{MN} in order to compute
\begin{align*}
{\rm mult}_{\delta^2} = e_\delta \circ e_\delta &= \frac{1}{4} \Big( \mathfrak{q}_2 \mathfrak{q}_{-1} \mathfrak{q}_{-1} (\Delta_{123}) + \mathfrak{q}_1 \mathfrak{q}_1 \mathfrak{q}_{-2} (\Delta_{123}) \Big)^{\circ 2} \\
&= \frac{1}{4} \Big( \mathfrak{q}_2 \mathfrak{q}_{-1} \mathfrak{q}_{-1} \mathfrak{q}_1 \mathfrak{q}_1 \mathfrak{q}_{-2} (\Delta_{123} \Delta_{456}) + \mathfrak{q}_1 \mathfrak{q}_1 \mathfrak{q}_{-2} \mathfrak{q}_2 \mathfrak{q}_{-1} \mathfrak{q}_{-1} (\Delta_{123} \Delta_{456}) \Big).
\end{align*}
Using the commutation relations for Nakajima operators, the second summand inside the brackets equals $-2 \mathfrak{q}_1 \mathfrak{q}_1 \mathfrak{q}_{-1} \mathfrak{q}_{-1} (\Delta_{1234})$. For the first summand, we obtain
\begin{align*}
\mathfrak{q}_2 \mathfrak{q}_{-1} \mathfrak{q}_{-1} \mathfrak{q}_1 \mathfrak{q}_1 \mathfrak{q}_{-2} (\Delta_{123} \Delta_{456}) &= -2 \mathfrak{q}_2 \mathfrak{q}_{-1} \mathfrak{q}_1 \mathfrak{q}_{-2} (\Delta_{1234}) \\
&= 2 \mathfrak{q}_2 \mathfrak{q}_{-2} \Big( (\rho_{14})_\ast(\Delta_{23} \cdot \Delta_{1234}) \Big) \\
&= 48 \mathfrak{q}_2 \mathfrak{q}_{-2} (c_1 c_2),
\end{align*}
where in the last step we used $\Delta_{1234} = \Delta_{12} \Delta_{23} \Delta_{34}$ and $\Delta_S^2 = 24 c_1 c_2$ from~\eqref{Identities} as well as $\Delta_S c_1 = \Delta_S c_2 = c_1 c_2$, concluding the proof.
\end{proof}

\begin{prop}\label{LNakajima}
Let $S$ be a projective $K3$ surface and $F = S^{[2]}$. Then, as correspondences in ${\rm CH}^2(F \times F)$, we can express $L$ in terms of Nakajima operators by
\begin{equation}
L = - \mathfrak{q}_1 \mathfrak{q}_{-1}([S \times S]) - \frac{1}{8} \mathfrak{q}_2 \mathfrak{q}_{-2}([S \times S]) - \mathfrak{q}_1 \mathfrak{q}_1 \mathfrak{q}_{-1} \mathfrak{q}_{-1}(c_1 + c_4).
\end{equation}
\end{prop}

\begin{proof}
We express each summand of~\eqref{LHilb} in terms of Nakajima operators. For $I$ we have $I = {^t\mathcal{Z}} \circ \mathcal{Z} =  - \mathfrak{q}_1 \mathfrak{q}_{-1}([S \times S])$ by \cite[Lemma~11.2]{SV} and Corollary~\ref{INakajima}. Next,
\begin{align*}
\delta = e_\delta([F]) &= \frac{1}{2} \Big( \mathfrak{q}_2 \mathfrak{q}_{-1} \mathfrak{q}_{-1} (\Delta_{123}) + \mathfrak{q}_1 \mathfrak{q}_1 \mathfrak{q}_{-2} (\Delta_{123}) \Big) \circ \frac{1}{2} \mathfrak{q}_1 \mathfrak{q}_1 ([S \times S]) \cdot 1_{S^{[0]}} \\
&= \frac{1}{2} \mathfrak{q}_2 ([S]) \cdot 1_{S^{[0]}}.
\end{align*}
Together with Corollary~\ref{GeorgsTrick} this gives
\begin{align*}
- \frac{1}{2} \delta_1 \delta_2 &= - \frac{1}{8} \Big( \mathfrak{q}_2([S]) \cdot 1_{S^{[0]}} \Big) \boxtimes \Big( \mathfrak{q}_2([S]) \cdot 1_{S^{[0]}} \Big) \\
&= - \frac{1}{8} \mathfrak{q}_2([S]) \mathfrak{q}_2'([S]) \cdot \Delta_{S^{[0]}} \\
&= - \frac{1}{8} \mathfrak{q}_2 \mathfrak{q}_{-2} ([S \times S]).
\end{align*}
Similarly, using Proposition~\ref{mult}, we get $S_c = {\rm mult}_{S_c}([F]) = \mathfrak{q}_1 \mathfrak{q}_1 (c_1) \cdot 1_{S^{[0]}}$.
Applying again Corollary~\ref{GeorgsTrick} yields
\begin{equation*}
-2(S_c)_1 = - \mathfrak{q}_1 \mathfrak{q}_1 \mathfrak{q}_{-1} \mathfrak{q}_{-1} (c_4) \ \ \text{ and } \ \ -2(S_c)_2 = - \mathfrak{q}_1 \mathfrak{q}_1 \mathfrak{q}_{-1} \mathfrak{q}_{-1} (c_1),
\end{equation*}
and putting everything together gives the claimed formula for $L$.
\end{proof}

On $S^{[2]}$ recall Oberdieck's canonical lifts \cite[eq.~(5)~and~(6)]{O}
\begin{equation}
-2 \sum_{n \geq 1} \frac{1}{n^2} \mathfrak{q}_n \mathfrak{q}_{-n} (a_1 + a_2) = -2 \mathfrak{q}_1 \mathfrak{q}_{-1}(a_1 + a_2) - \frac{1}{2} \mathfrak{q}_2 \mathfrak{q}_{-2}(a_1 + a_2)
\end{equation}
of~$\widetilde{f}_a$ if $a$ is in the surface part ${\rm CH}^1(S) \subseteq {\rm CH}^1(S^{[2]}) = {\rm CH}^1(S) \oplus \mathbb{Q} \delta$, and
\begin{align}
\begin{split}
&- \frac{1}{3} \sum_{i+j+k=0} :\mathfrak{q}_i \mathfrak{q}_j \mathfrak{q}_k (\frac{1}{k^2} \Delta_{12} + \frac{1}{j^2} \Delta_{13} + \frac{1}{i^2} \Delta_{23} + \frac{2}{j \cdot k} c_1 + \frac{2}{i \cdot k} c_2 + \frac{2}{i \cdot j} c_3): \\
&=  2 \mathfrak{q}_2 \mathfrak{q}_{-1} \mathfrak{q}_{-1} \left( c_3 -c_1 - \Delta_{12} - \frac{1}{8} \Delta_{23} \right) + 2\mathfrak{q}_1 \mathfrak{q}_1 \mathfrak{q}_{-2} \left( c_1 - c_3 - \Delta_{23} - \frac{1}{8} \Delta_{12} \right),
\end{split}
\end{align}
of $\widetilde{f}_{\delta}$. Recall also the lift
\begin{equation}
2 \sum_{n \geq 1} \frac{1}{n} \mathfrak{q}_n \mathfrak{q}_{-n} (c_2 - c_1) = 2 \mathfrak{q}_1 \mathfrak{q}_{-1} (c_2 - c_1) + \mathfrak{q}_2 \mathfrak{q}_{-2} (c_2 - c_1).
\end{equation}
of $h$.

\begin{prop}\label{T_a=f_a}
Let $F = S^{[2]}$ and $a \in {\rm CH}^1(F)$. Let $\widetilde{F}_a$ be as in Corollary~\ref{T_a} and $H$ as in Proposition~\ref{H} with respect to the canonical lift $L$ from Theorem~\ref{LdefHilb}. Then both cycles agree with Oberdieck's canonical lifts.
\end{prop}

\begin{proof}
First, let $a$ be in the surface part ${\rm CH}^1(S) \subseteq {\rm CH}^1(S^{[2]})$. We begin by expressing $l_1 a_1$ in terms of Nakajima operators, $l_2 a_2$ being its transpose. We use
\begin{equation*}
l_1 a_1 = [F \times F] \circ {\rm mult}_l \circ {\rm mult}_a
\end{equation*}
as correspondences and the formulas from Proposition~\ref{mult} as well as the formula for $e_a = {\rm mult}_a$ from~\cite[eq.~(4)]{O}. We have
\begin{align}\label{l_1a_1}
\begin{split}
l_1 a_1 &= \mathfrak{q}_1 \mathfrak{q}_1 \mathfrak{q}_{-1} \mathfrak{q}_{-1}([S^4]) \ \circ \\
&\circ \left( 5 \mathfrak{q}_1 \mathfrak{q}_{-1}(c_1 c_2) - \frac{5}{16}  \mathfrak{q}_1 \mathfrak{q}_1 \mathfrak{q}_{-1} \mathfrak{q}_{-1}(\Delta_{1234}) \right) \circ  \mathfrak{q}_1 \mathfrak{q}_{-1}(\Delta_\ast(a)) \\
&= \mathfrak{q}_1 \mathfrak{q}_1 \mathfrak{q}_{-1} \mathfrak{q}_{-1}([S^4]) \circ \left( 5 \mathfrak{q}_1 \mathfrak{q}_1 \mathfrak{q}_{-1} \mathfrak{q}_{-1}(\Delta_{14} a_1 c_2 c_3) + \frac{5}{8} \mathfrak{q}_1 \mathfrak{q}_1 \mathfrak{q}_{-1} \mathfrak{q}_{-1}(\Delta_{1234} a_1) \right).
\end{split}
\end{align}
For the first summand in \eqref{l_1a_1}, we get
\begin{equation*}
5 \mathfrak{q}_1 \mathfrak{q}_1 \mathfrak{q}_{-1} \mathfrak{q}_1 \mathfrak{q}_{-1} \mathfrak{q}_1 \mathfrak{q}_{-1} \mathfrak{q}_{-1}(\Delta_{48} a_8 c_6 c_7) - 5 \mathfrak{q}_1 \mathfrak{q}_1 \mathfrak{q}_{-1} \mathfrak{q}_1 \mathfrak{q}_{-1} \mathfrak{q}_{-1}(c_4 c_5 a_6) = 10 \mathfrak{q}_1 \mathfrak{q}_1 \mathfrak{q}_{-1} \mathfrak{q}_{-1} (c_3 a_4).
\end{equation*}
The second summand in \eqref{l_1a_1} equals
\begin{equation*}
\frac{5}{8} \mathfrak{q}_1 \mathfrak{q}_1 \mathfrak{q}_{-1} \mathfrak{q}_1 \mathfrak{q}_{-1} \mathfrak{q}_1 \mathfrak{q}_{-1} \mathfrak{q}_{-1} (\Delta_{4678} a_8) - \frac{5}{8} \mathfrak{q}_1 \mathfrak{q}_1 \mathfrak{q}_{-1} \mathfrak{q}_1 \mathfrak{q}_{-1} \mathfrak{q}_{-1} (\Delta_{456} a_6) = \frac{5}{2} \mathfrak{q}_1 \mathfrak{q}_1 \mathfrak{q}_{-1} \mathfrak{q}_{-1} (c_3 a_4),
\end{equation*}
where we used $\Delta_{34} a_3 = a_3 c_4 + c_3 a_4$ (see~\cite{BV}). Hence, $l_1 a_1 = \frac{25}{2} \mathfrak{q}_1 \mathfrak{q}_1 \mathfrak{q}_{-1} \mathfrak{q}_{-1} (c_3 a_4)$, and therefore
\begin{equation*}
\frac{4}{25} (l_1 a_1 + l_2 a_2) = 2 \mathfrak{q}_1 \mathfrak{q}_1 \mathfrak{q}_{-1} \mathfrak{q}_{-1} (a_1 c_2 + c_3 a_4).
\end{equation*}
We now calculate $L a_1 = (-L) \circ (-e_a)$ employing the formula for $L$ from Proposition~\ref{LNakajima}. We have
\begin{align*}
L a_1 &= \left( \mathfrak{q}_1 \mathfrak{q}_{-1}([S \times S]) + \frac{1}{8} \mathfrak{q}_2 \mathfrak{q}_{-2}([S \times S]) + \mathfrak{q}_1 \mathfrak{q}_1 \mathfrak{q}_{-1} \mathfrak{q}_{-1}(c_1 + c_4) \right) \\
&\circ \Big( \mathfrak{q}_1 \mathfrak{q}_{-1}(\Delta_\ast(a)) + \mathfrak{q}_2 \mathfrak{q}_{-2} (\Delta_\ast(a)) \Big) \\
&= \mathfrak{q}_1 \mathfrak{q}_{-1} \mathfrak{q}_1 \mathfrak{q}_{-1} (\Delta_{34} a_4) + \frac{1}{8} \mathfrak{q}_2 \mathfrak{q}_{-2} \mathfrak{q}_2 \mathfrak{q}_{-2} (\Delta_{34} a_4) + \mathfrak{q}_1 \mathfrak{q}_1 \mathfrak{q}_{-1} \mathfrak{q}_{-1} \mathfrak{q}_1 \mathfrak{q}_{-1} ((c_1 + c_4) \Delta_{56} a_6) \\
&= - \mathfrak{q}_1 \mathfrak{q}_{-1} (a_2) - \frac{1}{4} \mathfrak{q}_2 \mathfrak{q}_{-2} (a_2) + \mathfrak{q}_1 \mathfrak{q}_1 \mathfrak{q}_{-1} \mathfrak{q}_{-1} (a_1 c_4 - c_1 a_4 - c_3 a_4).
\end{align*}
After adding $L a_2$, clearly $\mathfrak{q}_1 \mathfrak{q}_1 \mathfrak{q}_{-1} \mathfrak{q}_{-1} (a_1 c_4 - c_1 a_4)$ cancels out. We obtain
\begin{equation*}
2L(a_1 + a_2) = - 2 \mathfrak{q}_1 \mathfrak{q}_{-1} (a_1 + a_2) - \frac{1}{2} \mathfrak{q}_2 \mathfrak{q}_{-2} (a_1 + a_2) - 2 \mathfrak{q}_1 \mathfrak{q}_1 \mathfrak{q}_{-1} \mathfrak{q}_{-1} (a_1 c_2 + c_3 a_4),
\end{equation*}
and putting both pieces together yields the desired formula. Let now $a = \delta$. Using the formula for $e_\delta$ from~\cite[eq.~(4)]{O}, we can analogously compute
\begin{equation*}
l_1 \delta_1 = [F \times F] \circ {\rm mult}_l \circ e_\delta = - 25 \mathfrak{q}_1 \mathfrak{q}_1 \mathfrak{q}_{-2} (c_3),
\end{equation*}
so that
\begin{equation}\label{ld}
\frac{4}{25} (l_1 \delta_1 + l_2 \delta_2) = -4 \mathfrak{q}_1 \mathfrak{q}_1 \mathfrak{q}_{-2} (c_3) - 4 \mathfrak{q}_2 \mathfrak{q}_{-1} \mathfrak{q}_{-1} (c_1).
\end{equation}
Using the formula for $L$, we now consider $2 L \delta_1 = 2 (-L) \circ (- e_\delta)$ and get
\begin{equation*}
2 L \delta_1 = -2 \mathfrak{q}_1 \mathfrak{q}_1 \mathfrak{q}_{-2} (\Delta_{23}) - \frac{1}{4} \mathfrak{q}_2 \mathfrak{q}_{-1} \mathfrak{q}_{-1} (\Delta_{23}) + 2 \mathfrak{q}_1 \mathfrak{q}_1 \mathfrak{q}_{-2} (c_1 + c_3),
\end{equation*}
where we applied the Nakajima commutation relations several times. Therefore,
\begin{align}\label{Ld}
\begin{split}
2L(\delta_1 + \delta_2) &= \mathfrak{q}_1 \mathfrak{q}_1 \mathfrak{q}_{-2} \left( 2(c_1 + c_3) - 2 \Delta_{23} - \frac{1}{4} \Delta_{12} \right) \\
&+ \mathfrak{q}_2 \mathfrak{q}_{-1} \mathfrak{q}_{-1} \left( 2(c_1 + c_3) - 2 \Delta_{12} - \frac{1}{4} \Delta_{23} \right).
\end{split}
\end{align}
Putting equations~\eqref{ld} and~\eqref{Ld} together yields the claimed formula for~$\widetilde{F}_\delta$. The proof for $H$ is very similar and uses the decomposition of the small diagonal in~\eqref{Identities} as well as the equation $\mathfrak{q}_1 \mathfrak{q}_{-1} (\Delta_S) = - {\rm id}_{S^{[1]}}$ as correspondences on $S^{[1]} \times S^{[1]}$.
\end{proof}

\begin{rmk}\label{lastRmk}
Proposition~\ref{T_a=f_a} together with the main theorem of \cite{O} shows the equation $(a,a)H = \widetilde{H}_a$ (see Remark~\ref{conjRelRmk}) for all divisor classes $a$ in the Hilbert scheme case and so yields further evidence that this might be true for hyperkähler varieties $X$ of $K3^{[2]}$-type in general. It also gives another proof for Conjecture~\ref{multiplicativityConj} for Hilbert schemes of two points of $K3$~surfaces by the more general \cite[Theorem~1.4]{NOY}.
\end{rmk}

\appendix
\section{An auxiliary lemma}\label{app}

\begin{lem}\label{LemmaGeorgsTrick}
Let $X$, $Y$ and $Z$ be smooth projective varieties and $\Gamma \in {\rm CH}^\ast(X \times Y)$, $\widetilde{\Gamma} \in {\rm CH}^\ast(X \times Z)$ correspondences. Then as correspondences in ${\rm CH}^\ast(Y \times Z)$ we have
\begin{equation*}
(\Gamma \times \widetilde{\Gamma})_\ast(\Delta_X) = \widetilde{\Gamma} \circ {^t \Gamma}.
\end{equation*}
\end{lem}

\begin{proof}
Writing out the definitions of both sides, we see that we only need to show
\begin{equation}\label{toshow}
p_{X_1 Y}^\ast( \Gamma ) \cdot p_{X_2 Z}^\ast( \widetilde{\Gamma} ) \cdot p_{X_1 X_2}^\ast(\Delta_X) = ({\rm id}_Y \times \Delta_X \times {\rm id}_Z)_\ast \left( p_{YX}^\ast({^t \Gamma}) \cdot p_{XZ}^\ast(\widetilde{\Gamma}) \right).
\end{equation}
In here, $p_{YX}$ and $p_{XZ}$ are the projections from the triple product $Y \times X \times Z$ to $Y \times X$ and $X \times Z$, respectively. The projections $p_{X_1 Y}$, $p_{X_2 Z}$ and $p_{X_1 X_2}$ are the projections from the product $X \times Y \times X \times Z$ to the factors indicated by the indices. Moreover, $\Delta_X$ denotes both the diagonal embedding $X \hookrightarrow X \times X$ and the cycle class of its image. We can now rewrite the argument of the right hand side of~\eqref{toshow} as
\begin{align*}
p_{YX}^\ast({^t \Gamma}) \cdot p_{XZ}^\ast(\widetilde{\Gamma}) &= ({\rm id}_Y \times \Delta_X \times {\rm id}_Z)^\ast \left( p_{Y X_1}^\ast({^t \Gamma}) \right) \cdot ({\rm id}_Y \times \Delta_X \times {\rm id}_Z)^\ast \left( p_{X_2 Z}^\ast(\widetilde{\Gamma}) \right) \\
&= ({\rm id}_Y \times \Delta_X \times {\rm id}_Z)^\ast \left( p_{Y X_1}^\ast({^t \Gamma}) \cdot p_{X_2 Z}^\ast(\widetilde{\Gamma}) \right).
\end{align*}
Hence, by the projection formula, the entire right hand side of~\eqref{toshow} equals
\begin{align*}
({\rm id}_Y \times \Delta_X \times {\rm id}_Z)_\ast \left( p_{YX}^\ast({^t \Gamma}) \cdot p_{XZ}^\ast(\widetilde{\Gamma}) \right) &= p_{Y X_1}^\ast({^t \Gamma}) \cdot p_{X_2 Z}^\ast(\widetilde{\Gamma}) \cdot p_{X_1 X_2}^\ast(\Delta_X) \\
&= p_{X_1 Y}^\ast(\Gamma) \cdot p_{X_2 Z}^\ast(\widetilde{\Gamma}) \cdot p_{X_1 X_2}^\ast(\Delta_X).
\end{align*}
\end{proof}

This seemingly trivial result has an interesting consequence in the case $X = {\rm Spec}(\mathbb{C})$ where $\Delta_X$ is an isomorphism. The following was pointed out to me by my advisor Georg Oberdieck and uses only the fact that the Nakajima correspondences $\mathfrak{q}_i$ and $\mathfrak{q}_{-i}$ are the transpose of each other up to sign.

\begin{cor}\label{GeorgsTrick}
Let $S$ be a smooth projective surface and recall the definition of the Nakajima operators on Hilbert schemes of points of surfaces. Denote by $\mathfrak{q}_i$ and $\mathfrak{q}_j'$ the operators
\begin{equation*}
\mathfrak{q}_i, \mathfrak{q}_j': \bigoplus_{m,n} {\rm CH}^\ast(S^{[m]} \times S^{[n]}) \rightarrow \bigoplus_{m,n} {\rm CH}^\ast(S^{[m]} \times S^{[n]} \times S)
\end{equation*}
acting as the Nakajima operator $\mathfrak{q}_i$ on the first factor and as the identity on the second factor respectively as $\mathfrak{q}_j$ on the second factor and as the identity on the first factor. Obviously, $\mathfrak{q}_i$ commutes with $\mathfrak{q}_j'$. Let $\Gamma \in {\rm CH}^\ast(S^{k+l})$. Then we have the equation of correspondences
\begin{equation*}
\mathfrak{q}_{i_1} \cdots \mathfrak{q}_{i_k} \mathfrak{q}_{j_1}' \cdots \mathfrak{q}_{j_l}' (\Gamma) \cdot 1_{S^{[0]}} = (-1)^{i_1 + \ldots + i_k} \mathfrak{q}_{j_1} \cdots \mathfrak{q}_{j_l} \mathfrak{q}_{- i_k} \cdots \mathfrak{q}_{-i_1} ( \tau_\ast(\Gamma) ),
\end{equation*}
where $\tau: S^{k+l} \rightarrow S^{k+l}$ permutes the factors according to the permutation of the indices.
\end{cor}

\begin{proof}
This follows from the definition of the Nakajima operators and Lemma~\ref{LemmaGeorgsTrick} on noting that $\Delta_{S^{[0]}} = 1_{S^{[0]}}$ under the identification $S^{[0]} \times S^{[0]} = S^{[0]}$ via the diagonal map.
\end{proof}

\medskip
\bibliographystyle{alpha}
\bibliography{MA_references}
\end{document}